\documentclass{amsart}
\usepackage{amsmath,amssymb,amsxtra,amsthm}

\newtheorem{theorem}{Theorem}
\newtheorem{corollary}{Corollary}
\newtheorem{lemma}{Lemma}
\newtheorem{definition}{Definition}
\newtheorem{example}{Example}
\newtheorem{remark}{Remark}
\newtheorem{question}{Question}
\newtheorem{app}{Application}

\begin{document}
\title{Completely multiplicative functions taking values in $\{-1,1\}$}
\author{Peter Borwein}%
\author{Stephen K.K. Choi}%
\author{Michael Coons}%
\thanks{Research supported in part by grants from NSERC of Canada and MITACS.}%
\address{Department of Mathematics, Simon Fraser University, B.C., Canada V5A 1S6}
\email{pborwein@cecm.sfu.ca, kkchoi@math.sfu.ca, mcoons@sfu.ca}
\date{June 12, 2008}
%%%%%%%%%%%%%%%%%%%%%%%%%%%%%%%%%%%%%%%%%%%%%%%%%%%%%%%%%%%%%%%%%%%

\begin{abstract} Define {\em the Liouville function for $A$}, a subset of the primes $P$, by $\lambda_{A}(n) =(-1)^{\Omega_A(n)}$ where $\Omega_A(n)$ is the number of prime factors of $n$ coming from $A$ counting multiplicity. For the traditional Liouville function, $A$ is the set of all primes. Denote $$L_A(n):=\sum_{k\leq n}\lambda_A(n)\quad\mbox{and}\quad R_A:=\lim_{n\to\infty}\frac{L_A(n)}{n}.$$ We show that for every $\alpha\in[0,1]$ there is an $A\subset P$ such that $R_A=\alpha$. Given certain restrictions on $A$, asymptotic estimates for $\sum_{k\leq n}\lambda_A(k)$ are also given. With further restrictions, more can be said. For {\em character--like functions} $\lambda_p$ ($\lambda_p$ agrees with a Dirichlet character $\chi$ when $\chi(n)\neq 0$) exact values and asymptotics are given; in particular $$\quad\sum_{k\leq n}\lambda_p(k)\ll \log n.$$ Within the course of discussion, the ratio $\phi(n)/\sigma(n)$ is considered.\end{abstract}

\maketitle

%%%%%%%%%%%%%%%%%%%%%%%%%%%%%%%%%%%%%%%%%%%%%%%%%%%%%%%%%%%%%%%%%%%

%%%%%%%%%%%%%%%%%%%%%%%%%%%%%%%%%%%%%%%%%%%%%%%%%
\section{Introduction}
%%%%%%%%%%%%%%%%%%%%%%%%%%%%%%%%%%%%%%%%%%%%%%%%%

Let $\Omega(n)$ be the number of distinct prime factors in $n$ (with multiple factors counted multiply). The Liouville $\lambda$--function is defined by
$$\lambda(n) :=(-1)^{\Omega(n)}.$$
So $\lambda(1)=\lambda(4)=\lambda(6) = \lambda(9) =\lambda(10) =1$ and $\lambda(2) = \lambda(5)=\lambda(7) =\lambda(8) = -1.$ In particular, $\lambda (p)=-1$ for any prime $p$. It is well-known (e.g. See \S 22.10 of \cite{HW}) that $\Omega$ is completely additive, i.e, $\Omega (mn)=\Omega (m)+\Omega (n)$ for any $m$ and $n$ and hence $\lambda$ is completely multiplicative, i.e., $\lambda (mn)=\lambda (m)\lambda (n)$ for all $m,n\in\mathbb{N}$. It is interesting to note that on the set of square-free positive integers $\lambda(n)=\mu(n)$, where $\mu$ is the M\"obius function. In this respect, the Liouville $\lambda$--function can be thought of as an extension of the M\"obius function.

Similar to the M\"obius function, many investigations surrounding the $\lambda$--function concern the summatory function of initial values of $\lambda$; that is, the sum %
$$L(x):=\sum_{n\leq x} \lambda(n).$$
Historically, this function has been studied by many mathematicians, including Liouville, Landau, P\'olya, and Tur\'an. Recent attention to the summatory function of the M\"obius function has been given by Ng \cite{Ng2, Ng1}. Larger classes of completely multiplicative functions have been studied by Granville and Soundararajan \cite{GS3,GS1,GS2}.

One of the most important questions is that of the asymptotic order of $L(x)$; more formally, the question is to determine the smallest value of $\vartheta$ for which $$\lim_{x\to\infty}\frac{L(x)}{x^{\vartheta}}=0.$$ It is known that the value of $\vartheta=1$ is equivalent to the prime number theorem \cite{Lan2,Lan3} and that $\vartheta=\frac{1}{2}+\varepsilon$ for any arbitrarily small positive constant $\varepsilon$ is equivalent to the Riemann hypothesis \cite{BCRW} (The value of $\frac{1}{2}+\varepsilon$ is best possible, as $\limsup_{x\to\infty}L(x)/\sqrt{x}>.061867$, see Borwein, Ferguson, and Mossinghoff \cite{BFM1}). Indeed, any result asserting a fixed $\vartheta\in \left(\frac{1}{2},1\right)$ would give an expansion of the zero-free region of the Riemann zeta function, $\zeta(s)$, to $\Re (s)\geq \vartheta$.

Unfortunately, a closed form for determining $L(x)$ is unknown. This brings us to the motivating question behind this investigation: {\em are there functions similar to $\lambda$, so that the corresponding summatory function does yield a closed form?}

Throughout this investigation $P$ will denote the set of all primes. As an analogue to the traditional $\lambda$ and $\Omega$, define {\em the Liouville function for $A\subset P$} by
$$\lambda_{A}(n) =(-1)^{\Omega_A(n)}$$ %
where $\Omega_A(n)$ is the number of prime factors of $n$ coming from $A$ counting multiplicity. Alternatively, one can define $\lambda_{A}$ as the completely multiplicative function with $\lambda_{A}(p) = -1$ for each prime $p \in A$ and $\lambda_A(p)=1$ for all $p\notin A$. Every completely multiplicative function taking only $\pm 1$ values is built this way. The class of functions from $\mathbb{N}$ to $\{-1,1\}$ is denoted $\mathcal{F}(\{-1,1\})$ (as in \cite{GS1}). Also, define $$L_A:=\sum_{n\leq x}\lambda_A(n)\quad\mbox{and}\quad R_A:=\lim_{n \to \infty} \frac{ L_A(x)}{n}.$$

In this paper, we first consider questions regarding the properties of the function $\lambda_A$ by studying the function $R_A$. The structure of $R_A$ is determined and it is shown that for each $\alpha\in[0,1]$ there is a subset $A$ of primes such that $R_A=\alpha$. The rest of this paper considers an extended investigation on those functions in $\mathcal{F}(\{-1,1\})$ which are character--like in nature (meaning that they agree with a real Dirichlet character $\chi$ at nonzero values). Within the course of discussion, the ratio $\phi(n)/\sigma(n)$ is considered.

%%%%%%%%%%%%%%%%%%%%%%%%%%%%%%%%%%%%%%%%%%%%%%%%%
\section{Properties of $L_A(x)$}
%%%%%%%%%%%%%%%%%%%%%%%%%%%%%%%%%%%%%%%%%%%%%%%%%

Define the {\em generalized Liouville sequence} as
$$\mathfrak{L}_A:=\{\lambda_{A}(1), \lambda_{A}(2), \ldots \}.$$
%
%and the {\em generalized Liouville number} as
%
%\begin{equation}\label{L} l_A:=  \sum_{n=1}^\infty \frac{\frac{1}{2}(\lambda_{A}(n)+1)}{2^n},
%\end{equation}
%
%where the numerator of the summands has the effect of mapping $-1\mapsto 0$ and $1\mapsto 1$.

%Questions about the sequence $\mathfrak{L}_A$ can be phrased more conveniently as questions about the number $l_A$.

\begin{theorem} The sequence $\mathfrak{L}_A$ is not eventually periodic.
\end{theorem}

\begin{proof} Towards a contradiction, suppose that $\mathfrak{L}_A$ is eventually periodic, say the sequence is periodic after the $M$--th term and has period $k$. Now there is an $N\in\mathbb{N}$ such that for all $n\geq N$, we have $nk>M$. Since $A\neq \varnothing$, pick $p\in A$. Then $$\lambda_A(pnk)=\lambda_A(p)\cdot\lambda_A(nk)=-\lambda_A(nk).$$ But $pnk\equiv nk (\mbox{mod}\ k)$, a contradiction to the eventual $k$--periodicity of $\mathfrak{L}_A$. \end{proof}

\begin{corollary} If $A\subset P$ is nonempty, then $\lambda_A$ is not a Dirichlet character.
\end{corollary}

\begin{proof} This is a direct consequence of the non--periodicity of $\mathfrak{L}_A$.
\end{proof}

%Since we now have that $l_A\notin\mathbb{Q}$, a fundamental question arises, which at present we are unable to answer in full generality: is $l_A$ algebraic or transcendental? For a partial answer to this question see \cite{BC1} and \cite{BC2}.

%%%%%%%%%%%%%%%%%%%%%%%%%%%%%%%%%%%%%%%%%%%%%%%%%
%\section{Partial sums of $\lambda_A$}
%%%%%%%%%%%%%%%%%%%%%%%%%%%%%%%%%%%%%%%%%%%%%%%%%

To get more acquainted with the sequence $\mathfrak{L}_A$, we study the partial sums $L_A(x)$ of $\mathfrak{L}_A$, and to study these, we consider the Dirichlet series with coefficients $\lambda_A(n)$.

Starting with singleton sets $\{p\}$ of the primes, a nice relation becomes apparent; for $\Re(s)>1$
\begin{equation}\label{zetasum}
\frac{(1-p^{-s})}{(1+p^{-s})}\zeta(s)=  \sum_{n=1}^\infty \frac{\lambda_{\{p\}}(n)}{n^s},
\end{equation}
and for sets $\{p,q\}$,
\begin{equation}\label{zetasum2}
\frac{(1-p^{-s})(1-q^{-s})}{(1+p^{-s})(1+q^{-s})} \zeta(s)=  \sum_{n=1}^\infty \frac{\lambda_{\{p,q\}}(n)}{n^s}.
\end{equation}

For any subset $A$ of primes, since $\lambda_A$ is completely multiplicative, for $\Re (s) > 1$ we have %
\begin{align} \mathcal{L}_A(s)&:=\sum_{n=1}^\infty \frac{\lambda_{A}(n)}{n^s}
 =  \prod_{p} \left( \sum_{l=0}^\infty \frac{\lambda_{A}(p^l)}{p^{ls}} \right)  \nonumber \\
& =  \prod_{p\in A} \left( \sum_{l=0}^\infty \frac{(-1)^l}{p^{ls}}\right)\prod_{p\not\in A} \left( \sum_{l=0}^\infty \frac{1}{p^{ls}}\right)  =  \prod_{p\in A} \left( \frac{1}{1+\frac{1}{p^s}}\right) \prod_{p\not\in A} \left( \frac{1}{1-\frac{1}{p^s}}\right)  \nonumber \\
& =  \zeta (s)  \prod_{p\in A} \left( \frac{1-p^{-s}}{1+ p^{-s}}\right) . \label{zetasum3}
\end{align}

This relation leads us to our next theorem, but first let us recall a vital piece of notation from the introduction.

\begin{definition} For $A\subset P$ denote
$$R_A:=\lim_{n \to \infty} \frac{ \lambda_{A}(1)+\lambda_{A}(2)+\ldots+\lambda_{A}(n)}{n}.$$
\end{definition}

The existence of the limit $R_A$ is guaranteed by Wirsing's Theorem. In fact, Wirsing in \cite{Wi} showed more generally that every real multiplicative function $f$ with $|f(n)|\le 1$ has a mean value, i.e, the limit
$$\lim_{x\to\infty}\frac{1}{x}\sum_{n \le x}f(n)$$
exists. Furthermore, in \cite{Win} Wintner showed that %
$$\lim_{x\to\infty}\frac{1}{x}\sum_{n \le x}f(n) = \prod_{p} \left( 1+\frac{f(p)}{p} + \frac{f(p^2)}{p^2} +\cdots \right) \left( 1-\frac{1}{p}\right) \neq 0$$
if and only if $\sum_{p}|1-f(p)|/p$ converges; otherwise the mean value is zero. This gives the following theorem.

\begin{theorem} For the completely multiplicative function $\lambda_A(n)$, the limit $R_A$ exists and
\begin{equation}\label{ra}R_A=\begin{cases}\prod_{p\in A}\frac{p-1}{p+1} & \mbox{ if $\sum_{p \in A}p^{-1} < \infty$,} \\0 & \mbox{ otherwise.}\end{cases}
\end{equation}
\end{theorem}

\begin{example} For any prime $p$, $R_{\{p\}} = \frac{p-1}{p+1}.$
\end{example}

To be a little more descriptive, let us make some notational comments. Denote by $\mathcal{P}(P)$ the power set of the set of primes. Note that $$\frac{p-1}{p+1}=1-\frac{2}{p+1}.$$ Recall from above that $R:\mathcal{P}(P)\to\mathbb{R}$, is defined by $$R_A:=\prod_{p\in A}\left(1-\frac{2}{p+1}\right).$$ It is immediate that $R$ is bounded above by 1 and below by 0, so that we need only consider that $R:\mathcal{P}(P)\to[0,1]$. It is also immediate that $R_\varnothing=1$ and $R_P=0$.

\begin{remark}For an example of a subset of primes with mean value in $(0,1)$, consider the set $K$ of primes defined by $$K:=\left\{p_n\in P: p_n=\min_{q>n^3}\{q\in P\}\mbox{ for } n\in\mathbb{N}\right\}.$$ Since there is always a prime in the interval $(x,x+x^{5/8}]$ (see Ingham \cite{Ing1}), these primes are well defined; that is, $p_{n+1}>p_n$ for all $n\in\mathbb{N}$. The first few values give $$K=\{ 11, 29, 67, 127, 223, 347, 521, 733, 1009, 1361,\ldots\}.$$ Note that $$\frac{p_n-1}{p_n+1}>\frac{n^3-1}{n^3+1},$$ so that $$R_k=\prod_{p\in K}\left(\frac{p-1}{p+1}\right)\geq\prod_{n=2}^\infty\left(\frac{n^3-1}{n^3+1}\right)=\frac{2}{3}.$$ Also $R_K<(11-1)/(11+1)=5/6,$ so that $$\frac{2}{3}\leq R_K<\frac{5}{6},$$ and $R_K\in(0,1)$.
\end{remark}

There are some very interesting and important examples of sets of primes $A$ for which $R_A=0$. Indeed, results of von Mangoldt \cite{vMan1} and Landau \cite{Lan2,Lan3} give the following equivalence.

\begin{theorem} The prime number theorem is equivalent to $R_P=0$.
\end{theorem}

We may be a bit more specific regarding the values of $R_A$, for $A\in\mathcal{P}(P)$. We will show that for each $\alpha\in (0,1)$, there is a set of primes $A$ such that $$R_A=\prod_{p\in A}\left(\frac{p-1}{p+1}\right)=\alpha.$$

\begin{lemma}\label{allp} Let $p_n$ denote the $n$th prime. For all $k\in\mathbb{N}$, $R_{[k,\infty)}=0.$
\end{lemma}

\begin{proof} Let $A=P\cap [k,\infty)$. For any $x \ge k$, we have
$$\sum_{\substack{p \le x \\ p\in A}} \frac{1}{p} = \sum_{k \le p \le x} \frac{1}{p}  = \log\log x + O_k(1).$$ Since this series diverges, so $R_A=0$ by \eqref{ra}.
\end{proof}

\begin{theorem}\label{allga} The function $R:\mathcal{P}(P)\to[0,1]$ is surjective. That is, for each $\alpha\in [0,1]$ there is a set of primes $A$ such that $R_A=\alpha.$ \end{theorem}

\begin{proof} Note first that $R_P=0$ and $R_\varnothing=1$. To prove the statement for the remainder of the values, let $\alpha\in(0,1)$. Then since %
$$\lim_{p\to\infty}R_{\{p\}}=\lim_{p\to\infty} \left(1-\frac{2}{p+1}\right)=1,$$
there is a minimal prime $q_1$ such that %
$$R_{\{q_1\}}=\left(1-\frac{2}{q_1+1}\right)>\alpha\qquad $$
i.e.,
$$
\frac{1}{\alpha}\cdot R_{\{q_1\}}=\frac{1}{\alpha}\left(1-\frac{2}{q_1+1}\right)>1.
$$
Similarly, for each $N\in\mathbb{N}$, we may continue in the same fashion, choosing $q_i>q_{i-1}$ (for $i=2\ldots N$) minimally, we have %
$$\frac{1}{\alpha}\cdot R_{\{q_1,q_2,\ldots,q_N\}}=\frac{1}{\alpha}\prod_{i=1}^N\left(1-\frac{2}{q_i+1}\right)>1.$$ %
Now consider %
$$\lim_{N\to\infty} \frac{1}{\alpha}\cdot R_{\{q_1,q_2,\ldots,q_N\}}=\frac{1}{\alpha}\prod_{i=1}^\infty\left(1-\frac{2}{q_i+1}\right),$$ %
where the $q_i$ are chosen as before. Denote $A=\{q_i\}_{i=1}^\infty.$ We know that %
$$\frac{1}{\alpha}\cdot R_A=\frac{1}{\alpha}\prod_{i=1}^\infty\left(1-\frac{2}{q_i+1}\right)\geq 1.$$

We claim that $R_A=\alpha$. To this end, let us suppose to the contrary that %
$$\frac{1}{\alpha}\cdot R_A=\frac{1}{\alpha}\prod_{i=1}^\infty\left(1-\frac{2}{q_i+1}\right)>1.$$ %
Applying Lemma \ref{allp}, we see that $P\backslash A$ is infinite (here $P$ is the set of all primes). As earlier, since %
$$\lim_{\substack{p\to\infty\\ p\in A\backslash P}}R_{\{p\}}=\lim_{p\to\infty} \left(1-\frac{2}{p+1}\right)=1,$$ %
there is a minimal prime $q\in A\backslash P$ such that %
$$\frac{1}{\alpha}\cdot R_A\cdot R_{\{q\}}=\frac{1}{\alpha}\left[\prod_{i=1}^\infty\left(1-\frac{2}{q_i+1}\right)\right]\cdot \left(1-\frac{2}{q+1}\right)>1.$$ %
Since $q$ is a prime and $q\notin A$, there is an $i\in\mathbb{N}$ with $q_i<q<q_{i+1}$. This contradicts that $q_{i+1}$ was a minimal choice. Hence %
$$\frac{1}{\alpha}\cdot R_A=\frac{1}{\alpha}\prod_{i=1}^\infty\left(1-\frac{2}{q_i+1}\right)=1,$$ %
and there is a set $A$ of primes such that $R_A=\alpha$.
\end{proof}

The following theorem gives asymptotic formulas for the mean value of $\lambda_A$ if certain condition on the density of $A$ in $P$ is assumed.

\begin{theorem}\label{R_A char} Suppose $A$ be a subset of primes with density
\begin{equation}\label{1.2}\sum_{\substack{p\le x \\ p\in A}}\frac{\log p}{p} = \frac{1-\kappa}{2}\log x + O (1)\end{equation} and $-1 \le \kappa \le 1$.

If $0 < \kappa \le 1$, then  we have $$\sum_{n \le x}\frac{\lambda_A(n)}{n} = c_{\kappa} (\log x)^\kappa + O(1)$$ and $$\sum_{n \le x} \lambda_A(n) = (1+o(1))c_{\kappa}\kappa x(\log x)^{\kappa -1},$$ where
\begin{equation}\label{4} c_{\kappa} = \frac{1}{\Gamma (\kappa +1)}\prod_p \left( 1-\frac{1}{p}\right)^\kappa \left( 1-\frac{\lambda_A(p)}{p} \right)^{-1} .
\end{equation}
In particular,
$$R_A=\lim_{x\to \infty} \frac{1}{x} \sum_{n \le x} \lambda_A(n) = \begin{cases} c_1=\prod_{p\in A} \left( \frac{p-1}{p+1} \right) & \mbox{ if $\kappa =1$,} \\ 0 & \mbox{ if $0 < \kappa < 1$.}\end{cases}
$$
Furthermore, $\mathcal{L}_A(s)$ has a pole at $s=1$ of order $\kappa$ with residue $c_\kappa \Gamma (\kappa +1)$, i.e., %
$$\mathcal{L}_A(s)=\frac{c_{\kappa}\Gamma (\kappa +1)}{(s-1)^{\kappa}} + \psi (s),\quad  \Re (s) > 1,$$
for some function $\psi (s)$ analytic on the region $\Re (s) \ge 1$. %
If $-1 \le \kappa < 0$, then $\mathcal{L}_A(s)$ has zero at $s=1$ of order $-\kappa$, i.e., %
$$\mathcal{L}_A(s)=\frac{\zeta (2s)}{c_{-\kappa}\Gamma (-\kappa +1)}(s-1)^{-\kappa}(1 + \varphi (s))$$ %
for some function $\varphi (s)$ analytic on the region $\Re (s)\ge 1$ and hence
$$\mathcal{L}_A(1)=\sum_{n=1}^\infty \frac{\lambda_A(n)}{n}=0$$
and %
$$R_A=\lim_{x\to \infty} \frac{1}{x} \sum_{n \le x} \lambda_A(n) = 0.$$
If $\kappa = 0$, then  $\mathcal{L}_A(s)$ has no pole nor zero at $s=1$. In particular, we have
$$\sum_{n=1}^\infty \frac{\lambda_A(n)}{n}=\alpha \neq 0$$
and
$$R_A=\lim_{x\to \infty} \frac{1}{x} \sum_{n \le x} \lambda_A(n) = 0.$$
\end{theorem}

The proof of Theorem \ref{R_A char} will require the following result.

\begin{theorem}[Wirsing] Suppose $f$ is a completely multiplicative function which satisfies
\begin{equation}\label{1.1}\sum_{n \le x}\Lambda (n) f(n) = \kappa \log x +O (1) \end{equation}
and
\begin{equation}\label{2}\sum_{n \le x}\left|f(n)\right| \ll \log x
\end{equation}
with $0 \le  \kappa \le 1$ where $\Lambda (n)$ is the von Mangoldt function. Then we have
\begin{equation}\label{3}\sum_{n \le x}f(n) = c_f(\log x)^\kappa + O (1)
\end{equation}
where
\begin{equation}\label{3.1}c_f := \frac{1}{\Gamma (\kappa +1)}\prod_p \left( 1-\frac{1}{p}\right)^\kappa \left(\frac{1}{1-f(p)}\right) \end{equation}
where $\Gamma (\kappa )$ is the Gamma function.
\end{theorem}

\begin{proof} This can be found in Theorem 1.1 at P.27 of \cite{IK} by replacing condition (1.89) by \eqref{2}.
\end{proof}

\begin{proof}[Proof of Theorem \ref{R_A char}] Suppose first that $0 < \kappa \le 1$. We choose $f(n)= \frac{\lambda_A(n)}{n}$ in Wirsing Theorem. Since %
$$\sum_{n \le x}\frac{\Lambda (n)}{n} = \sum_{p \le x}\frac{\log p}{p} + O(1) = \log x + O(1),$$ %
so
\begin{eqnarray*} \sum_{n\le x}\frac{\Lambda (n)}{n}\lambda_A(n)& = & \sum_{p\le x}\frac{\log p}{p}\lambda_A(p) +O \left( \sum_{p^l \le x, l \ge 2}\frac{\log p}{p^l}\right) \\
& = & \sum_{p\le x}\frac{\log p}{p}\lambda_A(p) +O \left( \sum_{n \le x}\frac{\Lambda (n)}{n} - \sum_{p \le x}\frac{\log p}{p} \right) \\
& = & \sum_{p\le x}\frac{\log p}{p}\lambda_A(p) +O (1).
\end{eqnarray*}
On the other hand, from \eqref{1.2} we have
\begin{eqnarray*}
\sum_{p\le x}\frac{\log p}{p}\lambda_A(p) & = & \sum_{p \le x}\frac{\log p}{p} -2\sum_{\substack{p \le x \\ p \in A}}\frac{\log p}{p} \\
& = & \kappa \log x + O (1).
\end{eqnarray*}
Hence we have
$$\sum_{n \le x}\frac{\Lambda (n)}{n}\lambda_A(n) = \kappa \log x +O (1)$$
and condition  \eqref{1.1} is satisfied.

It then follows from \eqref{3} and \eqref{4} that %
$$\sum_{n \le x}\frac{\lambda_A(n)}{n} = c_{\kappa}(\log x)^\kappa + O(1).$$
From \eqref{1.2}, we have
\begin{eqnarray*}
\mathcal{L}_A(s+1)=      \sum_{n=1}^\infty \frac{\lambda_A(n)}{n^{s+1}} & = & \int_1^\infty y^{-s} d\sum_{n \le y} \frac{\lambda_A(n)}{n} \\
& = & \int_1^\infty y^{-s} d\left( c_{\kappa} (\log y)^\kappa + O(1)\right) \\
& = & c_{\kappa}\kappa \int_1^\infty \frac{(\log y)^{\kappa -1}}{y^{s+1}}dy + \int_1^\infty y^{-s}d O (1) \\
& = & c_{\kappa} \Gamma (\kappa +1)s^{-\kappa} + \psi(s)
\end{eqnarray*}
for $\Re (s) >0$ because %
$$\int_1^\infty \frac{(\log y)^{\kappa -1}}{y^{s+1}}dy = \Gamma (\kappa )s^{-\kappa}.$$
Here $\psi (s)$ is an analytic function on $\Re (s)\ge 0$.

Therefore, $\mathcal{L}_A(s)$ has a pole at $s=1$ of order $0 < \kappa \le 1$. Now from a generalization of the Wiener-Ikehara theorem (e.g. Theorem 7.7 of \cite{BD}), we have
$$\sum_{n \le x} \lambda_A(n) = (1+o(1)) c_{\kappa}\kappa x(\log x)^{\kappa -1}$$
and hence %
$$R_A=\lim_{x\to \infty} \frac{1}{x} \sum_{n \le x} \lambda_A(n) = \begin{cases} c_1 & \mbox{ if $\kappa =1$,} \\
0 & \mbox{ if $0 < \kappa < 1$.}
\end{cases}$$
However, %
$$c_1= \prod_p \left( 1-\frac{1}{p}\right) \left( 1-\frac{\lambda_A(p)}{p} \right)^{-1} =\prod_{p\in A} \left( \frac{1-p^{-1}}{1+p^{-1}} \right).$$

If $-1 \le \kappa < 0$,  we denote the complement of $A$ by $\overline{A}$. Then we have
\begin{eqnarray*}
\mathcal{L}_{\overline{A}}(s) & = & \sum_{n=1}^\infty \frac{\lambda_{\overline{A}}(n)}{n^s} =  \zeta (s) \prod_{p\not\in A} \left(\frac{1-p^{-s}}{1+p^{-s}}\right) \\
& = & \frac{\zeta (2s)}{\zeta (s)}\prod_{p\in A} \left(\frac{1+p^{-s}}{1-p^{-s}}\right)  =  \frac{\zeta (2s)}{\mathcal{L}_A(s)}
\end{eqnarray*}
for $\Re (s) > 1$. Hence, for $\Re(s) > 0$,  we have
\begin{equation}\label{7}\mathcal{L}_{\overline{A}}(s)\mathcal{L}_A(s) = \zeta(2s).
\end{equation}

From \eqref{1.2}, we have
$$\sum_{\substack{p\le x \\ p\not\in A}}\frac{\log p}{p}=\sum_{p\le x}\frac{\log p}{p} - \sum_{\substack{p\le x \\ p\in A}}\frac{\log p}{p} = \frac{1+\kappa}{2} \log x + O (1)$$
and
$$\sum_{n \le x}\frac{\Lambda (n)}{n}\lambda_{\overline{A}}(n) = - \kappa \log x + O (1).$$

We then apply the above  case to $\mathcal{L}_{\overline{A}}(s)$ and deduce that $\mathcal{L}_{\overline{A}}(s)$ has a pole at $s=1$ of order $-\kappa$, then in view of \eqref{7}, $\mathcal{L}_A(s)$ has a zero at $s=1$ of order $-\kappa$, i.e.,
$$\mathcal{L}_A(s)=\frac{\zeta (2s)}{c_{-\kappa}\Gamma (-\kappa +1)}(s-1)^{-\kappa}(1 + \varphi (s))$$ %
for some function $\varphi (s)$ analytic on the region $\Re (s)\ge 1$. In particular, we have
\begin{equation}\label{7.1}\mathcal{L}_A(1)=\sum_{n=1}^\infty \frac{\lambda_A(n)}{n}=0.\end{equation}
This completes the proof of Theorem \ref{R_A char}.
\end{proof}

Recall that Theorem \ref{allga} tells us that any $\alpha\in[0,1]$ is a mean value of a function in $\mathcal{F}(\{-1,1\})$. The functions in $\mathcal{F}(\{-1,1\})$ can be put into two natural classes: those with mean value $0$ and those with positive mean value.

Asymptotically, those functions with mean value zero are more interesting, and it is in this class which the Liouville $\lambda$--function resides, and in that which concerns the prime number theorem and the Riemann hypothesis. We consider an extended example of such functions in Section \ref{glp}. Before this consideration, we ask some questions about those functions $f\in\mathcal{F}(\{-1,1,\})$ with positive mean value.

%%%%%%%%%%%%%%%%%%%%%%%%%%%%%%%%%%%%%%%%%%%%%%
\section{One question twice}
%%%%%%%%%%%%%%%%%%%%%%%%%%%%%%%%%%%%%%%%%%%%%%

It is obvious that if $\alpha\notin \mathbb{Q}$, then $R_A\neq \alpha$ for any finite set $A\subset P$. We also know that if $A\subset P$ is finite, then $R_A\in\mathbb{Q}$.

\begin{question} Is there a converse to this; that is, for $\alpha\in\mathbb{Q}$ is there a finite subset $A$ of $P$, such that $R_A=\alpha$?
\end{question}

The above question can be posed in a more interesting fashion. Indeed, note that for any finite set of primes $A$, we have that %
$$R_A=\prod_{p\in A} \frac{p-1}{p+1} =\prod_{p\in A} \frac{\phi (p)}{\sigma (p)}=\frac{\phi(z)}{\sigma(z)}$$
where $z=\prod_{p\in A}p$, $\phi$ is Euler's totient function and $\sigma$ is the sum of divisors function. Alternatively, we may view the finite set of primes $A$ as determined by the square--free integer $z$. In fact, the function $f$ from the set of square--free integers to the set of finite subsets of primes, defined by %
$$f(z)=f(p_1p_2\cdots p_r)=\{p_1,p_2,\ldots,p_r\},\qquad (z=p_1p_2\cdots p_r)$$
is bijective, giving a one--to--one correspondence between these two sets.

In this terminology, we ask the question as:

\begin{question} Is the image of $\phi(z)/\sigma(z):\{\mbox{square--free integers}\}\to\mathbb{Q}\cap(0,1)$ a surjection?
\end{question}

That is, for every rational $q\in(0,1)$, is there a square--free integer $z$ such that $\frac{\phi(z)}{\sigma(z)}=q\ ?$ As a start, we have Theorem \ref{allga}, which gives a nice corollary.

\begin{corollary} If $S$ is the set of square--free integers, then
$$\left\{ x\in\mathbb{R}: x=\lim_{\substack{k\to\infty\\ (n_k)\subset S}}\frac{\phi(n_k)}{\sigma(n_k)}\right\}=[0,1].$$
\end{corollary}

\begin{proof} Let $\alpha\in[0,1]$ and $A$ be a subset of primes for which $R_A=\alpha$. If $A$ is finite we are done, so suppose $A$ is infinite. Write %
$$A=\{a_1,a_2,a_3,\ldots\}$$ %
where $a_i<a_{i+1}$ for $i=1,2,3,\ldots$ and define $n_k=\prod_{i=1}^k a_i$. The sequence $(n_k)$ satisfies the needed limit.
\end{proof}

%%%%%%%%%%%%%%%%%%%%%%%%%%%%%%%%%%%%%%%%%%%%%%%%%%%%%%%%%%%%%%%%
\section{The functions $\lambda_p(n)$}\label{glp}
%%%%%%%%%%%%%%%%%%%%%%%%%%%%%%%%%%%%%%%%%%%%%%%%%%%%%%%%%%%%%%%%

We now turn our attention to those functions $\mathcal{F}(\{-1,1\})$ with mean value $0$; in particular, we wish to examine functions for which a sort of Riemann hypothesis holds: functions for which $\mathcal{L}_A(s)=\sum_{n\in\mathbb{N}}\frac{\lambda_A(n)}{n^s}$ has a large zero--free region; that is, functions for which $\sum_{n\leq x}\lambda_A(n)$ grows slowly.

To this end, let $p$ be a prime number. Recall that the Legendre symbol modulo $p$ is  defined as
$$ \left(\frac{q}{p}\right) =\begin{cases} 1 & \mbox{ if $q$ is a quadratic residue modulo $p$,} \\
-1 & \mbox{ if $q$ is a quadratic non-residue modulo $p$,}\\
0 & \mbox{ if $q\equiv 0 \pmod p$.} \end{cases}$$
Here $q$ is a quadratic residue modulo $p$ provided $q\equiv x^2 \pmod p$ for some $x\not\equiv 0 \pmod p$.

Define the function $\Omega_p(n)$ to be the number of prime factors, $q$, of $n$ with $\left(\frac{q}{p}\right)=-1$; that is, %
$$\Omega_p(n)=\#\left\{q:\mbox{$q$ is a prime, $q|n$, and $\left(\frac{q}{p}\right)=-1$}\right\}.$$

\begin{definition} The modified Liouville function for quadratic non-residues modulo $p$ is defined as
$$\lambda_p(n):=(-1)^{\Omega_p(n)}.$$
\end{definition}

Analogous to $\Omega (n)$, since $\Omega_p(n)$ counts primes with multiplicities, $\Omega_p(n)$ is completely additive, and so $\lambda_p(n)$ is completely multiplicative. This being the case, we may define $\lambda_p(n)$ uniquely by its values at primes.

\begin{lemma}\label{quad} The function $\lambda_p(n)$ is the unique completely multiplicative function defined by $\lambda_p(p)=1$, and for primes $q\neq p$ by $$\lambda_p(q)=\left(\frac{q}{p}\right).$$
\end{lemma}

\begin{proof} Let $q$ be a prime with $q|n$. Now $\Omega_p(q)=0$ or $1$ depending on whether $\left(\frac{q}{p}\right)=1$ or $-1$, respectively. If $\left(\frac{q}{p}\right)=1$, then $\Omega_p(q)=0$, and so $\lambda_p(q)=1$.

On the other hand, if $\left(\frac{q}{p}\right)=-1$, then $\Omega_p(q)=1$, and so $\lambda_p(q)=-1$. In either case, we have\footnote{Note that using the given definition $\lambda_p(p)=\left(\frac{p}{p}\right)=1.$} $$\lambda_p(q)=\left(\frac{q}{p}\right).$$ \end{proof}

Hence if $n=p^km$ with $p\nmid m$, then we have
\begin{equation}\label{def2} \lambda_p(p^km)=\left(\frac{m}{p}\right).\end{equation}

Similarly, we may define the function $\Omega_p'(n)$ to be the number of prime factors $q$ of $n$ with $\left(\frac{q}{p}\right)=1$; that is, %
$$\Omega_p'(n)=\#\left\{q:\mbox{$q$ is a prime, $q|n$, and $\left(\frac{q}{p}\right)=1$}\right\}.$$

Analogous to Lemma \ref{quad} we have the following lemma for $\lambda_p'(n)$ and theorem relating these two functions to the traditional Liouville $\lambda$-function.

\begin{lemma} The function $\lambda_p'(n)$ is the unique completely multiplicative function defined by $\lambda_p'(p)=1$ and for primes $q\neq p$, as %
$$\lambda_p'(q)=-\left(\frac{q}{p}\right).$$
\end{lemma}

\begin{theorem}\label{llp} If $\lambda(n)$ is the standard Liouville $\lambda$--function, then %
$$\lambda(n)=(-1)^k\cdot\lambda_p(n)\cdot \lambda_p'(n)$$ %
where $p^k\| n$, i.e., $p^k|n$ and $p^{k+1}\nmid n$.
\end{theorem}

\begin{proof} It is clear that the theorem is true for $n=1$. Since all functions involved are completely multiplicative, it suffices to show the equivalence for all primes. Note that $\lambda(q)=-1$ for any prime $q$. Now if $n=p$, then $k=1$ and %
$$(-1)^1\cdot\lambda_p(p)\cdot \lambda_p'(p)=(-1)\cdot (1)\cdot (1)=-1=\lambda(p).$$
If $n=q\neq p$, then
$$(-1)^0\cdot\lambda_p(q)\cdot \lambda_p'(q)=\left(\frac{q}{p}\right)\cdot\left(-\left(\frac{q}{p}\right)\right)=-\left(\frac{q^2}{p}\right)=-1=\lambda(q),$$
and so the theorem is proved.
\end{proof}

%%%%%%%%%%%%%%%%%%%%%%%%%%%%%%%%%%%%%%%%%%%%%%%%%%%%%%%%%%%%%%%%
%\section{The function $L_p(n)$}
%%%%%%%%%%%%%%%%%%%%%%%%%%%%%%%%%%%%%%%%%%%%%%%%%%%%%%%%%%%%%%%%

To mirror the relationship between $L$ and $\lambda$, denote by $L_p(n)$, the summatory function of $\lambda_p(n)$; that is, define $$L_p(n):=\sum_{k=1}^n \lambda_p(n).$$ It is quite immediate that $L_p(n)$ is not positive\footnote{For the traditional $L(n)$, it was conjectured by P\'olya that $L(n)\geq0$ for all $n$, though this was proven to be a non-trivial statement and ultimately false (See Haselgrove \cite{Has1}).} for all $n$ and $p$. To find an example we need only look at the first few primes. For $p=5$ and $n=3$, we have $$L_5(3)=\lambda_5(1)+\lambda_5(2)+\lambda_5(3)=1-1-1=-1<0.$$ Indeed, the next few theorems are sufficient to show that there is a positive proportion (at least $1/2$) of the primes for which $L_p(n)<0$ for some $n\in\mathbb{N}$.

\begin{theorem}\label{cor1} Let
$$n=a_0+a_1p+a_2p^2+\ldots +a_kp^k$$
be the base $p$ expansion of $n$, where $a_j \in \{ 0,1,2,\ldots ,p-1\}$. Then we have
\begin{equation}\label{2g}L_p(n):=\sum_{l=1}^n\lambda_p(l) = \sum_{l=1}^{a_0}\lambda_p(l)+\sum_{l=1}^{a_1}\lambda_p(l)+\ldots +\sum_{l=1}^{a_k}\lambda_p(l).\end{equation}
Here the sum over $l$ is regarded as empty if $a_j=0$.
\end{theorem}

Instead of giving a proof of Theorem \ref{cor1} in this specific form, we will prove a more general result for which Theorem \ref{cor1} is a direct corollary. To this end, let $\chi$ be a non-principal Dirichlet character modulo $p$ and for any prime $q$ let
\begin{equation}\label{f}
f(q):=\begin{cases} 1 & \mbox{ if $p=q$,} \\ \chi (q) & \mbox{ if $p\neq q$.} \end{cases}
\end{equation}
We extend $f$ to be a completely multiplicative function and get
\begin{equation}
\label{1}
f(p^lm)=\chi (m)
\end{equation}
for $l\ge 0$ and $p\nmid m$. \begin{theorem}
\label{thm 2.2}
Let $N(n,l)$ be the number of digits $l$ in the base $p$ expansion of $n$.
Then
\[
\sum_{j=1}^n f(j) =\sum_{l=0}^{p-1} N(n,l)\left(\sum_{m\le l}\chi (m) \right).
\]
\end{theorem}
\begin{proof}
We write the base $p$ expansion of $n$ as
\begin{equation}
\label{8}
n=a_0+a_1p+a_2p^2+\ldots + a_kp^k
\end{equation}
where $0 \le a_j \le p-1$.
We then observe that, by writing $j=p^lm$ with $p\nmid m$,
\[
\sum_{j=1}^n f(j) =  \sum_{l=0}^k\sum_{\substack{j=1 \\ p^l\| j}}^{n} f(j) =  \sum_{l=0}^k \sum_{\substack{m \le n/p^l\\ (m,p)=1}} f(p^lm).
\]
For simplicity, we write \[
A:=a_0+a_1p+\ldots +a_lp^l \hspace{5mm} \mbox{ and } \hspace{5mm} B:=a_{l+1}+a_{l+2}p+\ldots +a_kp^{k-l-1}
\]
so that $n=A+Bp^{l+1}$ in \eqref{8}.
It now follows from \eqref{1} and \eqref{8} that
\[
\sum_{j=1}^n f(j) = \sum_{l=0}^k \sum_{\substack{m \le n/p^l\\ (m,p)=1}} \chi (m)= \sum_{l=0}^k \sum_{m \le A/p^l+Bp} \chi (m) =\sum_{l=0}^k \sum_{m \le A/p^l} \chi (m) \]
because $\chi (p)=0$ and $\sum_{m=a+1}^{a+p} \chi (m) =0$ for any $a$. Now since \[
a_l \le A/p^l=(a_0+a_1p+\ldots +a_lp^l)/p^l < a_l +1
\]
so we have
\[
\sum_{j=1}^n f(j) = \sum_{l=0}^k \sum_{m \le a_l} \chi (m) = \sum_{l=0}^{p-1} N(n,l)\left(\sum_{m\le l}\chi (m) \right). \]
This proves the theorem.
\end{proof}

In this language, Theorem \ref{cor1} can be stated as follows.

\begin{corollary}\label{theorem 2.2} If $N(n,l)$ is the number of digits $l$ in the base $p$ expansion of $n$, then
\begin{equation}\label{*} L_p(n)=\sum_{j=1}^n \lambda_p(j) =\sum_{l=0}^{p-1} N(n,l)\left(\sum_{m\le l}\left(\frac{m}{p}\right) \right). \end{equation} \end{corollary}

As an application of this theorem consider $p=3$.

\begin{app}\label{cor2} The value of $L_3(n)$ is equal to the number of 1's in the base 3 expansion of $n$. \end{app}

\begin{proof}
Since $\left(\frac{1}{3}\right)=1$ and $\left(\frac{1}{3}\right)+\left(\frac{2}{3}\right)=0$, so if $n=a_0+a_13+a_23^2+\ldots +a_k3^k$
is the base $3$ expansion of $n$, then the right-hand side of \eqref{2g} (or equivalently, the right-hand side of \eqref{*}) is equal to $D_3(n)$. The result then follows from Theorem \ref{cor1} (or equivalently Corollary \ref{theorem 2.2}).
\end{proof}

Note that $L_3(n)=k$ for the first time when $n=3^0+3^1+3^2+\ldots + 3^k$ and is never negative. This is in stark contrast to the traditional $L(n)$, which is negative more often than not. Indeed, we may classify all $p$ for which $L_p(n)\geq 0$ for all $n\in\mathbb{N}$.

\begin{theorem}\label{thm3} The function $L_p(n)\geq 0$ for all $n$ exactly for those odd primes $p$ for which
$$\left(\frac{1}{p}\right)+\left(\frac{2}{p}\right)+\ldots +\left(\frac{k}{p}\right) \ge 0$$
for all $1\le k \le p$. \label{theorem3}
\end{theorem}

\begin{proof} We first observe from \eqref{def2} that if $0 \le r <p$, then %
$$\sum_{l=1}^r \lambda_p(l)=\sum_{l=1}^r \left(\frac{l}{p}\right).$$
From theorem \ref{cor1},
\begin{eqnarray*}
      \sum_{l=1}^n \lambda_p(l)
& = & \sum_{l=1}^{a_0}\lambda_p(l)+\sum_{l=1}^{a_1}\lambda_p(l)+\ldots
      +\sum_{l=1}^{a_k}\lambda_p(l) \\
& = & \sum_{l=1}^{a_0}\left(\frac{l}{p}\right)+\sum_{l=1}^{a_1}\left(\frac{l}{p}\right)+\ldots
      +\sum_{l=1}^{a_k}\left(\frac{l}{p}\right)
\end{eqnarray*}
because all $a_j$ are between $0$ and $p-1$. The result then follows.
\end{proof}

\begin{corollary} For $n\in\mathbb{N}$, we have
$$ 0 \leq L_3(n) \leq [ \log_3n] +1.$$
\end{corollary}
\begin{proof}
This follows from Theorem \ref{thm3}, Application \ref{cor2}, and the fact that the number of 1's in the base three expansion of $n$ is $\le [\log_3n]+1$.
\end{proof}

As a further example, let $p=5$.

\begin{corollary}\label{cor5} The value of $L_5(n)$ is equal to the number of 1's in the base 5 expansion of $n$ minus the number of 3's in the base 5 expansion of $n$.
Also for $n\ge 1$, %
$$ \left| L_5(n)\right|  \le [\log_5n]+1. $$
\end{corollary}

Recall from above, that $L_3(n)$ is always nonnegative, but $L_5(n)$ isn't. Also $L_5(n)=k$ for the first time when $n=5^0+5^1+5^2+\ldots + 5^k$ and $L_5(n)=-k$ for the first time when $n=3\cdot 5^0+3\cdot 5^1+3\cdot 5^2+\ldots + 3\cdot 5^k$.

\begin{remark}\label{remp}The reason for specific $p$ values in the proceeding two corollaries is that, in general, it's not always the case that $|L_p(n)|\le [\log_pn]+1$. 
\end{remark}

We now return to our classification of primes for which $L_p(n)\geq 0$ for all $n\geq 1$.

\begin{definition}  Denote by $\mathcal{L}^+$, the set of primes $p$ for which $L_p(n)\geq 0$ for all $n\in\mathbb{N}$.
\end{definition}

 We have found, by computation, that the first few values in $\mathcal{L}^+$ are
$$\mathcal{L}^+=\{3,7,11,23,31,47,59,71,79,83,103,131,151,167,191,199,239,251\ldots\}.$$
By inspection, $\mathcal{L}^+$ doesn't seem to contain any primes $p$, with $p\equiv 1\pmod 4$. This is not a coincidence, as demonstrated by the following theorem.

\begin{theorem} If $p\in \mathcal{L}^+$, then $p\equiv 3 \pmod 4$.
\end{theorem}

\begin{proof} Note that if $p\equiv 1\pmod 4$, then %
$$\left(\frac{a}{p}\right)=\left(\frac{-a}{p}\right)$$ %
for all $1\leq a\leq p-1$, so that %
$$\sum_{a=1}^{\frac{p-1}{2}}\left(\frac{a}{p}\right)=0.$$

Consider the case that $\left(\frac{(p-1)/2}{p}\right)=1$. Then %
$$\sum_{a=1}^{\frac{p-1}{2}}\left(\frac{a}{p}\right)=\sum_{a=1}^{\frac{p-1}{2}-1}\left(\frac{a}{p}\right)+\left(\frac{(p-1)/2}{p}\right)=\sum_{a=1}^{\frac{p-1}{2}-1}\left(\frac{a}{p}\right)+1,$$ %
so that %
$$\sum_{a=1}^{\frac{p-1}{2}-1}\left(\frac{a}{p}\right)=-1<0.$$

On the other hand, if $\left(\frac{(p-1)/2}{p}\right)=-1$, then since $\left(\frac{(p-1)/2}{p}\right)=\left(\frac{(p-1)/2+1}{p}\right)$, we have %
$$\sum_{a=1}^{\frac{p-1}{2}}\left(\frac{a}{p}\right)=\sum_{a=1}^{\frac{p-1}{2}+1}\left(\frac{a}{p}\right)-\left(\frac{(p-1)/2+1}{p}\right)=\sum_{a=1}^{\frac{p-1}{2}+1}\left(\frac{a}{p}\right)+1,$$ %
so that %
\begin{equation*}\sum_{a=1}^{\frac{p-1}{2}+1}\left(\frac{a}{p}\right)=-1<0.\qedhere\end{equation*}
\end{proof}

%%%%%%%%%%%%%%%%%%%%%%%%%%%%%%%%%%%%%%%%%%%%%%%%%%%%%%%%%%%%%%%%
\section{A bound for $|L_p(n)|$}
%%%%%%%%%%%%%%%%%%%%%%%%%%%%%%%%%%%%%%%%%%%%%%%%%%%%%%%%%%%%%%%%

Above we were able to give exact bounds on the function $|L_p(n)|$. As explained in Remark \ref{remp}, this is not always possible, though an asymptotic bound is easily attained with a few preliminary results.

\begin{lemma}\label{period} For all $r,n\in\mathbb{N}$ we have $L_p(p^rn)=L_p(n)$.
\end{lemma}

\begin{proof} For $i=1,\ldots,p-1$ and $k\in\mathbb{N}$, $\lambda_p(kp+i)=\lambda_p(i)$. This relation immediately gives for $k\in\mathbb{N}$ that $L_p(p(k+1)-1)-L_p(pk)=0$, since $L_p(p-1)=0$. Thus $$L_p(p^rn)=\sum_{k=1}^{p^{r}n} \lambda_p(k)=\sum_{k=1}^{p^{r-1}n} \lambda_p(pk)=\sum_{k=1}^{p^{r-1}n} \lambda_p(p)\lambda_p(k)=\sum_{k=1}^{p^{r-1}n} \lambda_p(k)=L_p(p^{r-1}n).$$ The lemma follows immediately.
\end{proof}

\begin{theorem}\label{Lmax} The maximum value of $|L_p(n)|$ for $n<p^i$ occurs at $n=k\cdot\sigma(p^{i-1})$ with value $$\max_{n<p^i}|L_p(n)|=i\cdot \max_{n<p}|L_p(n)|,$$ where $\sigma(n)$ is the sum of the divisors of $n$.
\end{theorem}

\begin{proof} This follows directly from Lemma \ref{period}.
\end{proof}

\begin{corollary}\label{Lpn} If $p$ is an odd prime, then $ |L_p(n)| \ll\log n;$ furthermore, $$\max_{n\leq x}|L_p(x)| \asymp\log x.$$
\end{corollary}

%%%%%%%%%%%%%%%%%%%%%%%%%%%%%%%%%%%%%%%%%%%%%%%%%%%%%%%%%%%%%%%%
%\section{Conclusion}
%%%%%%%%%%%%%%%%%%%%%%%%%%%%%%%%%%%%%%%%%%%%%%%%%%%%%%%%%%%%%%%%

%We have considered a variant of the Liouville $\lambda$--function, $\lambda_p(n)$, together with its summatory function, $L_p(n)$. Within this analysis, we have provided a connection between $\lambda_p(n)$ and the Legendre symbol. This connection was exploited throughout this investigation, for example, within the proof of Corollary \ref{theorem 2.2}, to give a way to compute $L_p(n)$ exactly, and Lemma \ref{period}, which leads to the asymptotic result $L_p(x)\ll \log x.$

%%%%%%%%%%%%%%%%%%%%%%%%%%%%%%%%%%%%%%%%%%%%%%%%%%%%%%%%%%%%%%%%%%%

%\nocite{*}
\bibliographystyle{amsplain}

\end{document}